\documentclass[12pt]{amsart}
\usepackage{geometry}   
\usepackage[colorlinks,citecolor = red, linkcolor=blue,hyperindex]{hyperref}
\usepackage{euscript,eufrak,verbatim, mathrsfs}
\usepackage[psamsfonts]{amssymb}
\usepackage[usenames]{color}
\usepackage{bbm}
\usepackage{graphicx}
 \usepackage{float}

 \usepackage[all, cmtip]{xy}

\usepackage{upref, xcolor, dsfont}
\usepackage{amsfonts,amsmath,amstext,amsbsy, amsopn,amsthm}
\usepackage{enumerate}

\usepackage{url}

\usepackage{mathtools}
\usepackage{bookmark}

 \usepackage{euscript}
\usepackage{helvet}         
\usepackage{courier}        
\usepackage{type1cm}        
\usepackage{multicol}        
\usepackage[bottom]{footmisc}

\newtheorem{theorem}{Theorem}[section]
\newtheorem*{theorem*}{Theorem B} 
\newtheorem{lemma}[theorem]{Lemma}

\newtheorem{proposition}[theorem]{Proposition}

\newtheorem*{definition*}{Definition}
\newtheorem*{remark*}{Remark}

\newtheorem*{observation*}{Observation}

\newtheorem*{assumption*}{Assumption}
\newtheorem*{question*}{Question}
\newtheorem{remark}[theorem]{Remark}

\newcommand{\N}{\mathbb{N}}

\newcommand{\C}{\mathbb{C}}
\newcommand{\E}{\mathbb{E}}
\newcommand{\PP}{\mathbb{P}}

\newcommand{\supp}{\mathrm{supp}}

\newcommand{\an}{\text{\, and \,}}

\begin{document}

\title[Superposition rigidity involving DPP]{A rigidity property of superpositions involving determinantal processes}

\author
{Yanqi Qiu}
\address
{Yanqi QIU: Institute of Mathematics and Hua Loo-Keng Key Laboratory of Mathematics, AMSS, Chinese Academy of Sciences, Beijing 100190, China; CNRS, Institut de Math\'ematiques de Toulouse, Universit\'e Paul Sabatier
}
\email{yanqi.qiu@amss.ac.cn}

\begin{abstract}
The main result of this paper states that if $(N, \Pi)$ is a pair of independent point processes on a common ground space with $N$ Poisson and $\Pi$ determinantal induced by a locally trace class (not necessarily self-adjoint) correlation kernel, then their independent superposition $N + \Pi$ determines uniquely the distributions of $N$ and $\Pi$.
\end{abstract}


\maketitle

\setcounter{equation}{0}

\section{Introduction}
Let $E$ be an arbitrary locally compact Polish space, equipped with a non-negative Radon measure $\mu$. Let $\N$ be the set of non-negative integers. Let $\Gamma = \Gamma(E)$ be the set of Radon measures on $E$ with values in $\N \cup \{\infty\}$, equipped with the smallest sigma-algebra $\mathcal{B}_\Gamma$ making the following mappings measurable: 
\[
\xi  \in \Gamma \mapsto \xi (B) \in \N \cap \{\infty\},
\]
where $B$ ranges over all Borel subsets of $E$. By a {\it point process} $\Lambda$ on $E$, we always mean a $\Gamma(E)$-valued random variable. Thus a point process is a random Radon counting measure. For further background on the general theory of point processes, see  Daley and Vere-Jones \cite{DV-1}, Kallenberg \cite{Kallenberg}.

The {\it independent superposition} of two {\it independent} point processes $\Lambda_1$ and $\Lambda_2$ on $E$ is defined as their sum
\[
\Lambda_1 + \Lambda_2.
\] 
The above sum is understood as the sum of two random counting measures on $E$.

Our main result is the following ridigity property of superpositions involving determinantal processes with locally trace class ({\it not necessarily self-adjoint}) kernels. The  definitions of Poisson and determinantal processes will be recalled in \S \ref{sec-def} and \S \ref{sec-def-dpp}.

\begin{theorem}\label{thm-main}
Let $(N, \Pi)$ be a pair of independent point processes on $E$ with $N$ Poisson and  $\Pi$ determinantal induced by a locally trace class (not necessarily self-adjoint) correlation kernel. Let $(N', \Pi')$ be another such independent pair. If  the independent superpositions of the two pairs have the same distribution:
\[
N + \Pi \stackrel{d}{=} N' +  \Pi',
\]  
then 
\[
N \stackrel{d}{=} N' \text{\, and \,} \Pi \stackrel{d}{=} \Pi'.
\]
\end{theorem}

\begin{remark}
In statistic mechanics, determinantal processes were originally developed in the study of the statistic mechanics of fermion particle systems. While Poisson point processes are usually used to model many random processes in nature such as locations of stars in the sky, defects in materials, chemical reactions etc. It would be interesting  to have a physical interpretation of the superposition rigidity stated in Theorem \ref{thm-main}. 
\end{remark}

A more general result, Theorem \ref{thm-gen}, is given after the proof of Theorem \ref{thm-main}.

\subsection{Outline  of the proof of Theorem \ref{thm-main}}
We will use the {\it probability generating functionals} of  Poisson and determinantal point processes. More precisely, the probability generating functionals of two independent point processes $N$ and $\Pi$ ($N$ being Poisson and $\Pi$ being determinantal), give rise to two families of {\it entire functions} on $\C$: for any bounded compactly supported function $\varphi: E\rightarrow\C$, define two entire functions by
\[
z \mapsto \mathscr{B}_N (z \varphi), \quad z \mapsto \mathscr{B}_\Pi(z \varphi), \quad z \in \C.
\]
See \eqref{exp-lin} and \eqref{pgf-dpp} for the definitions of $\mathscr{B}_N$ and $\mathscr{B}_\Pi$. 
 It turns out that for any fixed $\varphi$, the entire function $z \mapsto \mathscr{B}_N(z \varphi)$ never vanishes (Lemma \ref{lem-pgf-ppp}), while the entire function $z \mapsto \mathscr{B}_\Pi(z \varphi)$ is {\it uniquely determined by its zeros} (see Lemma \ref{lem-pgf-dpp} for precise statement).
Note that the probability generating functional of the independent superposition $N + \Pi$ is given by the product: 
\[
\mathscr{B}_{N + \Pi} (z \varphi)  =  \mathscr{B}_{N}(z \varphi) \cdot \mathscr{B}_\Pi(z \varphi). 
\] 
Now the following elementary observation is crucial for us.
\begin{itemize}
\item {\it It is possible to reconstruct two unknown entire functions from their product, provided that one of the two functions is uniquely determined by its zeros and  the other is non-vanishing.} 
\end{itemize} 
Using the above observation, we see that the probability generating functional of the independent superposition $N +  \Pi$ determines uniquely the probability generating functionals of $N$ and $\Pi$ and hence determines uniquely the distributions of $N$ and $\Pi$.

\subsection{Poisson point processes}\label{sec-def}

A point process $N$  on $E$ is called a {\it Poisson point process} with intensity measure $\nu$ (which is a Radon measure on $E$) if it satisfies the following properties: 
\begin{itemize}
\item For any finitely many disjoint relatively compact Borel subsets $B_1, \cdots, B_k \subset E$, the random variables $N(B_1), \cdots, N(B_k)$ are independent. 
\item For any relatively compact Borel subset $B \subset E$, the random variable $N(B)$ follows the Poisson distribution: 
\[
\PP (N(B) = n)  = e^{-\nu(B) } \cdot \frac{\nu(B)^n}{n!}, \text{ for all $n \in \N$}.
\]
\end{itemize}

 The {\it Bogoliubov functional} or the {\it probability generating functional} for the above Poisson point process $N$ is defined by
\begin{align}\label{exp-lin}
 \mathscr{B}_{N}(\varphi) : = \E \Big(\prod_{x \in E} (1 + \varphi(x))^{N(\{x\})}\Big),  \text{\, for all $\varphi \in \mathcal{B}_c (E)$}.
\end{align}
Here and after, $\mathcal{B}_c(E)$ denotes the set of bounded Borel functions $\varphi: E\rightarrow \C$ with compact support.

The following lemma is well-known, we refer to Kingman \cite[Section 3.2]{Kingman-PPP} for its proof. 

\begin{lemma}\label{lem-ppp-bog}
 Let $N$ be a Poisson point process on $E$ with intensity measure $\nu$. Then  
\begin{align}\label{ppp-bog}
\mathscr{B}_{N}(\varphi)  = \exp\Big(\int_E \varphi d\nu\Big),  \text{\, for all $\varphi \in \mathcal{B}_c (E)$}.
\end{align}
\end{lemma}

\subsection{Determinantal point processes}\label{sec-def-dpp}

 In what follows, if $g_1, g_2: E \rightarrow \C$ are bounded Borel functions and $K: L^2(E, \mu)\rightarrow L^2(E, \mu)$ is a bounded linear operator, then we define
\[
g_1 K g_2: = M_{g_1} \circ K \circ M_{g_2},
\]
where $M_g$ is the operator defined by $M_g(f)= gf$ for all $f \in L^2(E, \mu)$.

An operator $K: L^2(E, \mu)\rightarrow L^2(E, \mu)$ is said to be locally trace class if 
\[
\mathds{1}_A K \mathds{1}_B
\]
is in the ideal of trace class operators on $L^2(E, \mu)$ for any relatively compact Borel subsets $A, B \subset E$ (the reader is referred to Simon's book \cite{Simon-trace} for more details on trace class operators).

A point process $\Pi$ on $E$ is said to be determinantal if there exists a locally trace class operator $K: L^2(E, \mu) \rightarrow L^2(E, \mu)$ such that the probability generating functional of $\Pi$ can be expressed by
\begin{align}\label{pgf-dpp}
\mathscr{B}_{\Pi}(\varphi)  = \E \Big(\prod_{x \in E} (1 + \varphi(x))^{\Pi(\{x\})}\Big) = \det ( 1 + \varphi K \mathds{1}_{\supp_e(\varphi)}), \text{ for all $\varphi \in \mathcal{B}_c (E)$,}
\end{align}
where $\supp_e(\varphi)$ is the {\it essential support} of the function $\varphi$, i.e., $\supp_e(\varphi)$ is the smallest closed subset of $E$ such that $\varphi$ vanishes $\mu$-almost everywhere on the complementary set $(\supp_e(\varphi))^c$. Here the Fredholm determinant $\det ( 1 + \varphi K \mathds{1}_{\supp_e(\varphi)})$ is well-defined since the locally trace class implies that $\varphi K \mathds{1}_{\supp_e(\varphi)}$ is trace class. The reader is referred to Simon \cite{Simon-det} for more details on Fredholm determinants and to \cite[Theorem 2]{DPP-S}, \cite[Theorem 1.2]{ST-palm} or \cite{Buf-mul} and  \cite[formula (16)]{QB-cmp} for more details on the formula \eqref{pgf-dpp}.

\begin{remark}
Our definition of determinantal processes uses only the correlation kernels $K$ as  {\bf bounded linear operators} on $L^2(E, \mu)$ which are locally trace class. However, it is worthwhile to mention that, if one needs to use the correlation functions of determinantal processes, then a special choices of the kernels $K$ as a functions on $E\times E$ will be needed.
\end{remark}

By the Macchi-Soshnikov theorem \cite{DPP-M}, \cite{DPP-S},  any {\it positive self-adjoint contractive operator} $K: L^2(E, \mu) \rightarrow L^2(E,\mu)$ that is locally trace class gives a determinantal process. See also  Shirai and Takahashi  \cite{ST-palm, ST-DPP}. There exist also  determinantal point processes with non-selfadjoint correlation kernels.

Although there exist determinantal processes with {\it non locally trace class} correlation kernels (see Borodin-Okounkov-Olshanski \cite{BOO-jams} and  Lytvynov \cite{Lytvynov-J}), it seems that our method in this paper works only for locally trace class (not necessarily self-adjoint) correlation kernels.

\section{Proof}

We first recall two elementary facts:  
\begin{itemize}
\item For both the Poisson point process $N$ and the determinantal point process $\Pi$ on $E$, the probabillity generating functionals $\varphi \in \mathcal{B}_c(E) \mapsto \mathscr{B}_N(\varphi)$ and $\varphi \in \mathcal{B}_c(E) \mapsto\mathscr{B}_\Pi(\varphi)$ determine uniquely the the distributions of $N$ and $\Pi$. This statement holds for any point processes whose probability generating functionals are well-defined.
\item The probability generating functional for the independent superposition $N + \Pi$ is
\begin{align}\label{pgf-super}
\mathscr{B}_{N + \Pi} (\varphi) = \mathscr{B}_N (\varphi) \cdot \mathscr{B}_\Pi (\varphi), \text{\, for all $\varphi \in \mathcal{B}_c(E)$.}
\end{align}
\end{itemize}
Using the above two elementary facts, we see that Theorem \ref{thm-main} follows immediately from the following proposition.

\begin{proposition}\label{prop-inj}
Let $N, \Pi$ be two independent point processes on $E$ with $N$ Poisson and $\Pi$  eterminantal induced by locally trace class correlation kernel. Then the product functional $\varphi \in \mathcal{B}_c(E)\mapsto \mathscr{B}_N (\varphi) \cdot \mathscr{B}_\Pi (\varphi)$ determines uniquely the two functionals $\varphi \in \mathcal{B}_c(E)\mapsto \mathscr{B}_N (\varphi)$ and $\varphi  \in \mathcal{B}_c(E)\mapsto \mathscr{B}_\Pi (\varphi)$.
\end{proposition}

\begin{lemma}\label{lem-pgf-ppp}
Let $N$ be a Poisson point process on $E$ with intensity measure $\nu$. Then  for any fixed $\varphi \in \mathcal{B}_c(E)$, the function $z\mapsto \mathscr{B}_{N}(z\varphi)$ is a non-vanishing entire function.
\end{lemma}

\begin{proof}
By Lemma \ref{lem-ppp-bog}, for any $\varphi \in \mathcal{B}_c (E)$, we have 
\[
\mathscr{B}_{N}(z\varphi) =  \exp\Big( z \int_E \varphi d\nu\Big), \, z \in \C,
\]
which is indeed a non-vanishing entire function.
\end{proof}

\begin{lemma}\label{lem-pgf-dpp}
Let $\Pi$ be a determinantal process on $E$ with locally trace class correlation kernel $K$. Then for any fixed $\varphi \in \mathscr{B}_c(E)$, the function $z \mapsto   \mathscr{B}_{\Pi}(z\varphi)$ is entire and  is uniquely determined by its zeros. More precisely, denote 
\[
\mathcal{Z} [  \mathscr{B}_{\Pi}(z\varphi)]
\] 
the zeros (counting multiplicities) of the  function $z \mapsto \mathscr{B}_\Pi(z \varphi)$, then $0 \notin \mathcal{Z}(\mathscr{B}_\Pi(z \varphi)$ and 
\begin{align}\label{sim-prod}
\sum_{x \in \mathcal{Z} [  \mathscr{B}_{\Pi}(z\varphi) ]} | 1/x| < \infty \an  \mathscr{B}_{\Pi}(z\varphi) = \prod_{x \in \mathcal{Z} [  \mathscr{B}_{\Pi}(z\varphi) ]} (1 -z/x).
\end{align}
\end{lemma}

The proof of Lemma \ref{lem-pgf-dpp} relies on the following well-known property of the Fredholm determinant.
 
\begin{lemma}[see Simon {\cite[Theorems 3.3 \& 4.2]{Simon-det}}]\label{lem-fred}
Let $T$ be a trace class operator acting on a complex Hilbert space. Then $z \mapsto \det(1 + z T)$ is an entire function whose zeros are exactly 
\[
\Big\{ - \frac{1}{\lambda_i(T)}\Big\}_{i = 1}^{N(T)},
\]
where $\{\lambda_i(T)\}_{i=1}^{N(T)}$ are non-zero eigenvalues (counting multiplicities) of the trace class operator $T$ and we have 
\begin{align}\label{zero-fred}
\det(1 + z T) = \prod_{i = 1}^{N(T)} ( 1 + z \lambda_i(T)).
\end{align}
In particular, the entire function $\det(1 + z T)$ is uniquely determined by its zeros. 
 \end{lemma}

\begin{proof}[Proof of Lemma \ref{lem-pgf-dpp}]
By the definition formula \eqref{pgf-dpp} for the determinantal process, for any $\varphi \in \mathcal{B}_c(E)$, we have
\[
\mathscr{B}_{\Pi}(z \varphi)   = \det ( 1 +  z \varphi K \mathds{1}_{\supp_e(\varphi)}).
\]
The locally trace class assumption on $K$ implies that $\varphi K \mathds{1}_{\supp_e(\varphi)}$ is trace class. Therefore, by Lemma \ref{lem-fred}, the function 
\[
z \mapsto \mathscr{B}_\Pi(z \varphi) 
\]
is entire  and is uniquely determined by its zeros. The representation \eqref{sim-prod} follows from the formula \eqref{zero-fred}.
\end{proof}

\begin{proof}[Proof of Proposition  \ref{prop-inj}]
 Define for each fixed $\varphi \in \mathcal{B}_c(E)$ an entire function 
\[
F_\varphi(z): = \mathscr{B}_N(z \varphi) \cdot \mathscr{B}_\Pi(z \varphi), \, z \in \C.
\]
Note that $F_\varphi$ is determined by the product functional $\mathscr{B}_N \cdot \mathscr{B}_\Pi$.  Now by Lemma \ref{lem-pgf-ppp}, $ z \mapsto \mathscr{B}_N(z \varphi)$ is non-vanishing on $\C$. Therefore, the zeros (counting multiplicities) of the entire function $z \mapsto \mathscr{B}_\Pi (z \varphi)$ are exactly $\mathcal{Z}(F_\varphi)$, the zeros (counting multiplicities) of $F_\varphi$. By Lemma \ref{lem-pgf-dpp}, we have 
\[
\mathscr{B}_\Pi(z \varphi)  =  \prod_{x \in \mathcal{Z}(F_\varphi)} (1 - z/x).
\]
In particular, we have 
\[
\mathscr{B}_\Pi(\varphi) = \prod_{x \in \mathcal{Z}(F_\varphi)} (1 - 1/x)
\]
and hence 
\[
\mathscr{B}_N(\varphi) = \frac{F_\varphi(1) }{ \prod_{x \in \mathcal{Z}(F_\varphi)} (1 - 1/x)}.
\]
The proof of Proposition \ref{prop-inj} is complete. 
\end{proof}

\subsection{Further generalizations}
Using the idea in proof of  Theorem \ref{thm-main}, we can immediately generalize our main result in a more general setting. 

Let $\mathcal{C}$ be the class of point processes $\Lambda$ on $E$ such that the probability generating functional 
\[
\varphi \in \mathcal{B}_c(E) \mapsto \mathscr{B}_\Lambda(\varphi) : = \E\Big(\prod_{x\in E} (1 + \varphi(x))^{\Lambda(\{x\})}\Big)
\]
is well-defined. It is easy to see that a point process $\Lambda$ on $E$ belongs to the class $\mathcal{C}$ iff for any relatively comapct Borel subset $B\subset E$ and any $a > 0$, 
\begin{align}\label{C-criterion}
\E(a^{\Lambda(B)})<\infty.
\end{align}

Using routine argument, we can prove that for any $\mathcal{C}$-type point process $\Lambda$ and any $\varphi \in \mathcal{B}_c(E)$,  the function $z\mapsto \mathscr{B}_\Lambda(z \varphi)$ is entire.  Now let us define two sub-classes of $\mathcal{C}$:
\begin{itemize}
\item Let $\mathcal{C}_{non-vanishing}$ be the sub-class of $\mathcal{C}$ consisting of $\mathcal{C}$-type point processes $\Theta$'s such that for any $\varphi \in \mathcal{B}_c(E)$, the entire function 
\[
z \mapsto \mathscr{B}_\Theta (z \varphi)
\]
is non-vanishing on $\C$.
\item  Let $\mathcal{C}_{zero}$ be the sub-class of $\mathcal{C}$ consisting of $\mathcal{C}$-type point processes $\Xi$'s such that for any $\varphi \in \mathcal{B}_c(E)$, the entire function 
\[
z \mapsto \mathscr{B}_\Xi (z \varphi)
\]
is uniquely determined by its zeros. 
 \end{itemize}

\begin{theorem}\label{thm-gen}
Let $\Theta$ and $\Xi$ be two independent point processes on $E$, of type $\mathcal{C}_{non-vanishing}$ and type $\mathcal{C}_{zero}$ respectively. Then the independent superposition $\Theta + \Xi$ determines uniquely the distribution of $\Theta$, as well as that of $\Xi$. 
\end{theorem}

\begin{proof}
The proof is similar to that of Theorem \ref{thm-main}.
\end{proof}

It is worthwhile to note that the classes $\mathcal{C}_{non-vanishing}$ and $\mathcal{C}_{zero}$  are both closed under the operation of independent superposition. While the independent superpositions of Poisson point processes yield again Poisson point processes, the independent superpositions of determinantal point processes may yield point processes which are not determinantal.

 Let us give more examples of $\mathcal{C}_{non-vanishing}$-type  point processes. We call a point process $\Lambda$ infinitely divisible, if for any $n\in \N$, there exist independent identically distributed point processes $\Lambda_{1, n}, \cdots, \Lambda_{n,n}$, such that 
\begin{align}\label{n-superposition}
\Lambda \stackrel{d}{=} \Lambda_{1, n} + \cdots + \Lambda_{n, n}.
\end{align}

\begin{remark}\label{rem-inf-div}
The criterion \eqref{C-criterion} for the $\mathcal{C}$-type point processes implies that if $\Lambda$ is an infinite divisible $\mathcal{C}$-type point process, then for any $n \in \N$, the point processes $\Lambda_{1, n}, \cdots, \Lambda_{n,n}$ in the formula \eqref{n-superposition} all belong to the class $\mathcal{C}$. Therefore, for any non-negative function $\varphi \in \mathcal{B}_c(E)$,  we have 
\[
\mathscr{B}_{\Lambda_{1,n}}(\varphi) = \mathscr{B}_\Lambda(\varphi)^{1/n}, 
\]
which in turn implies that the distribution of $\Lambda_{1,n}$ is uniquely determined by that of $\Lambda$.
\end{remark}

\begin{proposition}
Any $\mathcal{C}$-type infinitely divisible point process is in the class $\mathcal{C}_{non-vanishing}$.
\end{proposition}

\begin{proof}
Let $\Lambda$ be a $\mathcal{C}$-type infinite divisible point process. Suppose by contradiction that $\Lambda$ does  not belong to the class $\mathcal{C}_{non-vanishing}$, then there exists $\varphi_0 \in \mathcal{B}_c(E)$ and $z_0\in \C$ such that 
\begin{align}\label{1-zero}
\mathscr{B}_\Lambda (z_0 \cdot\varphi_0) = 0.
\end{align}
Now since  $\Lambda$ is infinite divisible, by Remark \ref{rem-inf-div}, there exist  $\mathcal{C}$-type independent identically distributed point processes $\Lambda_{1,n}, \cdots, \Lambda_{n,n}$ such that the formula \eqref{n-superposition} holds. Therefore, we have 
\[
\mathscr{B}_\Lambda(z \cdot \varphi_0) = \mathscr{B}_{\Lambda_{1,n}} (z \cdot \varphi_0)^n.
\]
But then the equality \eqref{1-zero} implies that $\mathscr{B}_{\Lambda_{1,n}} (z_0 \cdot\varphi_0) = 0$. Therefore, $z_0$ is a zero of the entire function $z \mapsto \mathscr{B}_\Lambda(z\cdot \varphi_0)$ of multiplicity $\ge n$. Since $n$ is arbitrary, $z_0$ is a zero of the entire function  $z \mapsto \mathscr{B}_\Lambda(z\cdot \varphi_0)$ of infinite multiplicity and hence 
\[
\mathscr{B}_\Lambda (z \cdot \varphi_0) \equiv 0.
\]
This contradicts the trivial identity $\mathscr{B}_\Lambda (0) = 1$. 
\end{proof}

\begin{remark}
Any Cox process is infinite divisible. Moreover,  a Cox process $\Lambda$ on $E$ is in our class $\mathcal{C}$ iff its random intensity measure $\lambda$ satisfies: for any $a >0$, 
\[
\E \Big(a^{\lambda(B)} \Big) < \infty.
\]
\end{remark}

\subsection{Acknowledgement}
The author is deeply grateful to the anonymous referees for their valuable suggestions, especially  their suggestions on infinite divisible point processes, Cox processes and non-self adjoint correlation kernels for determinantal processes.



\end{document}